\documentclass{amsart}
\usepackage{amssymb,latexsym,enumerate,url}

\newtheorem{Thm}[equation]{Theorem}

\newtheorem{Lem}[equation]{Lemma}

\theoremstyle{remark}

\newtheorem*{Rem*}{Remark}
\theoremstyle{definition}

\newtheorem{Property}[equation]{Property}

\newtheorem*{Not*}{Notation}

\numberwithin{equation}{section}

\newcommand{\lt}{\mathord{<}}
\newcommand{\Lt}{\mathord{\prec}}
\begin{document}

\date{%
Sat Sep  5 10:26:59 CEST 2009}

\title[The number of left orders of a group]
{The space of left orders of a group is either finite or uncountable}

\author[P. A. Linnell]{Peter A. Linnell}
\address{Department of Mathematics \\
Virginia Tech \\
Blacksburg \\
VA 24061-0123 \\
USA}
\email{plinnell@math.vt.edu}
\urladdr{http://www.math.vt.edu/people/plinnell/}

\begin{abstract}
Let $G$ be a group and let $\mathcal{O}_G$ denote the set of left
orderings on $G$.  Then $\mathcal{O}_G$ can be topologized in a
natural way, and we shall study this topology to show that
$\mathcal{O}_G$ can never be countably infinite.
\end{abstract}

\keywords{left-ordered group, Cantor set}

\subjclass[2000]{Primary: 20F60; Secondary: 06F15}
\maketitle

\section{Introduction}

A group $G$ is left-ordered means that there is a total order $<$ on
the group $G$ which is left invariant.  Thus we have for all $g,x,y
\in G$ with $x \ne y$:
\begin{itemize}
\item
Either $x < y$ or $y < x$, but never both.
\item
If $g < x < y$, then $g < y$.
\item
If $x < y$, then $gx < gy$.
\end{itemize}

In this situation, we shall write $(G,<)$ to mean
the left-ordered group $G$ together with the given left order.  Of
course in general, there are many left orders on a left-ordered
group.  However when $G$ is finite, $G$ has no left orders except in
the case $G = 1$, and then $G$ has exactly one left order.

We study the space $\mathcal{O}_G$ of all left orders on the
left-ordered group $G$.  Throughout this paper we let
$\mathbb{N} = \{1,2,\dots\}$, the positive integers, and for $a,g
\in G$ we let $a^g = gag^{-1}$.
The topology on $\mathcal{O}_G$ is given by a base
of open sets of the form
\[
U_{g_1,\dots,g_n} := \{\lt \in \mathcal{O}_G
\mid g_1 < \dots < g_n\}
\]
where $g_i \in G$ and $n \in \mathbb{N}$.
Another way to describe this topology is that it is given by the
subbase $\{U_{1,g} \mid g \in G\setminus 1\}$.
Of course to check that a map is continuous, we need only
check that the inverse image of each element of
the subbase is open, in particular if $f \colon X \to \mathcal{O}_G$
is a map from a topological space $X$, then $f$ is continuous if and
only if $f^{-1} (U_{1,g})$ is open in $X$
for all $g \in G\setminus 1$.

For $g_1, \dots, g_n \in G$, set $V_{g_1,\dots,g_n} = U_{1,g_1} \cap
\dots \cap U_{1,g_n}$.  Then the sets $V_{g_1,\dots,g_n}$
($0 \le n \in \mathbb{Z}$) form a base
for the topology on $\mathcal{O}_G$.  Sometimes we will write
$V(G)_{g_1, \dots,g_n}$ for $V_{g_1,\dots,g_n}$ if it not clear with
which group we are working with.  An important property of
$\mathcal{O}_G$ is that it is a totally disconnected
compact Hausdorff space \cite[Theorem 1.4]{Sikora04}.  In the case
$G$ is countable, $\mathcal{O}_G$ is metrizable.  Also there is a
natural right $G$-action on $\mathcal{O}_G$ by homeomorphisms,
where for $g \in G$ and $\lt \in \mathcal{O}_G$, we define
\begin{equation} \label{Eaction}
x <_g y \Longleftrightarrow x^g < y^g
\end{equation}
for all $x,y \in G$.  Thus $\lt_{gh} = (\lt_g)_h$ for all $g,h \in G$.

Recently Dave Witte Morris \cite{Witte06} gave a fabulous proof that
a left-ordered amenable group is locally indicable.  His method used
the space $\mathcal{O}_G$ and some elementary ergodic theory.  It
would seem that $\mathcal{O}_G$ is worthy of further investigation.

If $(G,<)$ is a left-ordered group, then the positive cone of
$G$ (relative to $<$) is $P := \{g \in G \mid g>1\}$.
Then $P$ satisfies the following:
\begin{Property} \label{Pcone}
\begin{enumerate}[\normalfont(a)]
\item \label{Pconea}
If $g,h \in P$, then $gh \in P$.
\item \label{Pconeb}
If $g \in P$, then $g^{-1} \notin P$.
\item \label{Pconec}
If $1 \ne g \in G$, then either $g$ or $g^{-1}$ is in $P$.
\end{enumerate}
\end{Property}
Conversely given a subset $P$ of $G$ satisfying
\eqref{Pconea}, \eqref{Pconeb} and \eqref{Pconec} above,
one can define a left order $<$ on $G$ by $h < g \Leftrightarrow
h^{-1}g \in P$.  Also one can use $P$ to give $G$ a right order,
that is a total order $\prec$ which is right
invariant, by defining $h \prec g$ if and only if
$gh^{-1} \in P$.  From this it is easy to see that a group is
left-ordered if and only if it is right-ordered.  For convenience,
we only consider left-ordered groups.  If
$\lt_g = \lt$ for all $g \in G$, then $G$ is a bi-ordered group.

The purpose of this paper is to answer
\cite[Problem 16.51]{MazurovKhukhro06} in the negative.
The problem asks ``do there exist groups that can be right-ordered in
infinitely countably many ways?"  Considerable progress on this
problem was made by A.~V.~Zenkov in \cite{Zenkov97}, where he proved
that the number of right orders on a locally indicable group is
either finite or uncountable; see
\cite[Theorem 5.2.5]{KopytovMedvedev96}.  We shall prove
\begin{Thm} \label{Tuncountable}
There is no group which can be left-ordered in a countably infinite
number of ways.
\end{Thm}

A somewhat different proof of this theorem is
given in \cite[Theorem 3.1]{CNR09}.  Very recently, another proof has
been given in \cite[Proposition 3.9]{Clay09}.
Our proof will use the derived series (see \S \ref{Spreliminaries})
of $\mathcal{O}_G$ to reduce to the case $G$ is locally indicable,
and then to apply Zenkov's result above.
The structure of groups with a nonzero finite number of
left-orders is well-known; this is due to Tararin
\cite[Theorem 5.2.1]{KopytovMedvedev96}.  Further information on this
is given in \cite{Kirk06}.

I am very grateful to Adam S.~Sikora for noticing a bad error
in an earlier version of this paper \cite{Linnell06}; this paper
salvages the correct part of that paper.  I would also like to thank
Andrew Glass and Andr\'es Navas for encouraging me to write this
paper.

\section{Preliminaries} \label{Spreliminaries}

If $H$ is a subgroup of the group $G$,
then any left order on $G$ restricts to a left order on $H$, so we
have a well-defined restriction map $\rho_{G,H} \colon \mathcal{O}_G
\to \mathcal{O}_H$, which is clearly continuous, because
$\rho_{G,H}^{-1}(V(H)_h) = V(G)_h$ for all $h \in H$.

Let $X$ be an arbitrary Hausdorff topological space.  Then the
derived subset $X'$ of $X$ is the subset obtained from $X$ by
removing all its isolated points; equivalently $X'$ is the set of
limit points of $X$.  Then $X'$ is a closed subset of $X$.  Of
course, $X'$ itself can still have isolated points, so for each
ordinal $\alpha$, we define $X^{(\alpha)}$ by transfinite induction
as follows.
\begin{itemize}
\item
$X^{(0)} = X$.
\item
$X^{(\alpha+1)} = (X^{(\alpha)})'$.
\item
$X^{(\alpha)} = \bigcap_{\lambda < \alpha} X^{(\lambda)}$
if $\alpha$ is a limit ordinal.
\end{itemize}
It is clear that the subspaces $X^{(\alpha)}$ form a descending
sequence of closed subspaces of $X$.  For more details, see \cite[\S
8.5]{Semadeni71}; in particular by \cite[Theorem 8.5.2]{Semadeni71},
there is an ordinal $\alpha$ such that $X^{(\alpha+1)} =
X^{(\alpha)}$.  We require the following easy lemma.
\begin{Lem} \label{Lfinite}
Let $X$ be a nonempty countable compact Hausdorff space.  Then there
exists an ordinal $\alpha$ such that $X^{(\alpha)}$ is finite and
nonempty.
\end{Lem}
\begin{proof}
By \cite[Proposition 8.5.7]{Semadeni71}, let $\alpha$ be the least
ordinal such that $X^{(\alpha)} = \emptyset$.  Suppose $\alpha$
is a limit ordinal.  Then $X^{(\alpha)} =
\bigcap_{\lambda < \alpha} X^{(\lambda)}$.  Since the
$X^{(\lambda)}$ form a descending sequence of closed nonempty
subsets of the compact space $X$, we see that $\bigcap_{\lambda <
\alpha} X^{(\lambda)} \ne \emptyset$ and we have a contradiction.
Therefore $\alpha$ must be a successor ordinal and we may write
$\alpha = \beta+1$ for some ordinal $\beta$.  Then $X^{(\beta)}$ is a
nonempty compact Hausdorff space which consists only of isolated
points, because $(X^{(\beta)})' = \emptyset$.  Therefore
$X^{(\beta)}$ is finite and nonempty.
\end{proof}

\section{Proof of Theorem \ref{Tuncountable}}

\begin{proof}[Proof of Theorem \ref{Tuncountable}]
Suppose by way of contradiction that $G$ is a left-ordered group
such that $\mathcal{O}_G$ is countably infinite.  We have a
right $G$-action on $\mathcal{O}_G$ by
homeomorphisms defined by $x <_g y$
if and only if $x^g < y^g$ (see \eqref{Eaction}), and this
will restrict to $G$-actions on $\mathcal{O}_G^{(\alpha)}$ for all
ordinals $\alpha$.  By Lemma \ref{Lfinite}, there is an ordinal
$\beta$ such that $\mathcal{O}_G^{(\beta)}$ is finite
and nonempty, and then there
will be a normal subgroup $H$ of finite index in $G$ which fixes all
the elements of $\mathcal{O}_G^{(\beta)}$.
In particular, there is a
left order $\prec$ on $G$ such that $\Lt_h = \Lt$ for all $h \in
H$.  Clearly $\rho_{G,H}(\Lt)$ is a bi-order on $H$.  Now by
\cite[Lemma 2.3]{RhemtullaRolfsen02}, a bi-ordered group is of
``Conrad type" \cite[Section 2, p.~2570]{RhemtullaRolfsen02}
(cf.~\cite[Lemma 6.6.2(1,3), p.~121]{Glass99})
and \cite[Theorem 2.4]{RhemtullaRolfsen02} states
that if $\rho_{G,H} (\Lt)$ is of Conrad type and $H$ has finite index
in $G$, then $\prec$ is also of Conrad type.  Furthermore
by \cite[Theorem 4.1]{RhemtullaRolfsen02}, a group is
of Conrad type if and only if it is locally indicable.  We conclude
that $G$ is locally indicable, and the result now follows from
\cite[Theorem 5.2.5]{KopytovMedvedev96}.
\end{proof}

\bibliographystyle{plain}

\end{document}